\documentclass[12pt]{article}
\usepackage{amsmath,amsfonts,amsthm,amssymb,calligra}
\usepackage{times}
\usepackage{textcomp}
\newtheorem{theorem}{Theorem}

\newtheorem{proposition}{Proposition}

\theoremstyle{definition}
\newtheorem*{definition}{Definition}
\theoremstyle{remark}

\newcommand{\F}{{\mathbb{F}}}

\newcommand{\R}{{\mathbb{R}}}
\newcommand{\C}{{\mathbb{C}}}
\newcommand{\SC}{{\mathcal{S}}}
\newcommand{\LC}{{\mathcal{L}}}

\begin{document}

\title{Generalizations of Joints Problem}
\author{Ben Yang}
\date{}
\maketitle

\begin{abstract}
We generalize the joints problem to sets of varieties and prove almost sharp bound on the number of joints. As a special case, given a set of $N$ $2$-planes in $\R^6$, the number of points at which three $2$-planes intersect and span $\R^6$ is at most $CN^{3/2+\epsilon}$. We also get almost sharp bound on the number of joints with multiplicities. The main tools are polynomial partitioning and induction on dimension. 
\end{abstract}
\section{Introduction}

Various authors have considered the joints problem. It asks, given $N$ lines in $\R^3$, how many "joints" can there be, where a joint is a point at which at least three non-coplanar lines intersect. This problem first appeared in \cite{CEG+}, where it was proved that the number of joints is $O(N^{7/4})$. Later on progress was made in improving the bound by Sharir \cite{S}, Sharir and Welzl \cite{SW}, and Feldman and Sharir \cite{FS}. Wolff \cite{W} observed a connection between the problem of counting joints to the Kakeya problem. Bennett, Carbery and Tao \cite{BCT} exploited this connection and proved an upper bound conditioned on the angles at the joints. 

It has long been conjectured that the correct upper bound on the number of joints is $O(N^{3/2})$. matching the lower bound one can get by considering axis-parallel lines in a $\sqrt{N}\times \sqrt{N}\times \sqrt{N}$ grid. Guth and Katz \cite{GK1} settled this conjecture in the affirmative, showing that the number of joints (for lines in $\R^3$) is indeed $O(N^{3/2})$.
\begin{theorem}[Guth and Katz]
Any set of $N$ lines in $\R^3$ form $O(N^{\frac{3}{2}})$ joints.
\end{theorem}
 The proof was an adaptation of Dvir's argument in \cite{D} for the solution of the finite field Kakeya problem, which involves working with the zero set of a polynomial.

Since then generalizations of the joints problem have been studied by many authors. Quilodr\'{a}n \cite{Q} and Kaplan, Sharir and Shustin \cite{KSS} generalized the joints problem to all dimensions, showing that given $N$ lines in $\R^n$, there are at most $O(N^{\frac{n}{n-1}})$ points at which $n$ lines with linearly independent directions intersect. Their proofs also work for algebraic curves with uniformly bounded degree. It is also known that similar results hold if one replaces $\R$ with any field $\F$ and the proof is similar. 

In \cite{I1} and \cite{I2} Iliopoulou proved a multi-version of the joints problem. A point $x\in \F^n$ is a multijoint formed by the finite collection $\mathcal{L}_1,\dots ,\mathcal{L}_n$ of lines in $\F^n$ if there exist at least $n$ lines through $x$, one from each collection, spanning $\F^n$. She proved that there are $O((|\mathcal{L}_1|\cdots|\mathcal{L}_n|)^{\frac{1}{n-1}})$ multijoints for $\F=\R$ and any $n\geq 3$, and for any field $\F$ and $n=3$. The bounds for multijoints problem will imply the bounds for original joints problem if we simply take those $n$ sets of lines to be the same set. The result also generalizes to real algebraic curves. 

In \cite{I3} and \cite{I4} Iliopoulou also studied the problem of counting joints (and multijoints) with multiplicities. It is conjectured by Carbery (and possibly some other people) that, for any transversal collections $\mathcal{L}_1,\dots ,\mathcal{L}_n$ of $L_1,\dots ,L_n$ lines in $\R^n$ (transversal in the sense that, whenever $n$ lines, one from each collection, meet at a point, then they form a joint there), we have
\[
\sum\limits_{x}(\prod\limits_{1\leq i\leq n}N_i(x))^{\frac{1}{n-1}}=O((L_1\cdots L_n)^{\frac{1}{n-1}})
\] 
where $N_i(x)$ is the number of lines of $\mathcal{L}_i$ passing through $x$. The sum is taken over all joints (for a non-joint $x$, $\prod\limits_{1\leq i\leq n}N_i(x)$ would be $0$ anyway). This is a discrete analogue of Guth's endpoint multilinear Kakeya theorem. In \cite{I3} Iliopoulou essentially proved this conjecture in $\R^3$ and it seems hard to generalize the proof to higher-dimensional space due to the use of some facts from computational geometry that hold only in $\R^3$. Hablicsek \cite{H} proved the "non-multi-version" of this conjecture under the condition that the lines through a joint are in generic positions, but for arbitrary field.

So far to the best of the author's knowledge all the generalizations of joints problem are about lines or curves which are $1$-dimensional objects. The reason is that in all of those proofs polynomial method is used and one has to bound the number of intersections of a line (or an algebraic curve) with the zero set of a polynomial of certain degree when the line is not contained in the zero set. However, if we replace the line or curve with a higher-dimensional object (like a $2$-plane), the number of intersections of that object with the zero set of a polynomial could be infinite even if the object is not contained in the zero set.

In this paper we generalize the joints problem to higher-dimensional objects. The following is the definition of a joint (from now on when we say joint, we mean multijoint) formed by collections of varieties in $\R^n$. 

\begin{definition}
	Suppose $\SC _i$ is a set of $\alpha _i$-dimensional algebraic varieties in $\R ^n$, $1\leq i\leq k$, $k\geq 2$ and $\sum_{i=1}^{k} \alpha _i=n\geq 2$. A point $x$ is a joint if 
	\begin{enumerate}
		\item for each $i$ there exists $S_i \in \SC _i$ such that $x$ is a smooth point of $S_i$ ;
		\item the tangent spaces of $S_i$ at $x$ for all $i$ span $\R ^n$.
	\end{enumerate}
\end{definition}

The main result of this paper is the following almost sharp bound on the number of joints formed by sets of varieties:  

\begin{theorem}[Main Theorem]\label{main}
	Suppose $\SC _i$ is a non-empty set of $\alpha _i$-dimensional algebraic varieties in $\R ^n$, $1\leq i\leq k$, $k\geq 2$ and $\sum_{i=1}^{k} \alpha _i=n\geq 2$, $\alpha _i<n$. We also assume that each variety is defined by at most $m$ polynomial equations each of degree at most $d$. Then for any $\epsilon >0$, there exists a constant $C(n,m,d,\epsilon)$ such that the number of joints formed by those sets of varieties is bounded by $C(n,m,d,\epsilon)(\prod_{i=1}^{k} |\SC_i|)^{\frac{1}{k-1}+\epsilon}$.
\end{theorem}

In particular, it implies the following joints problem for $2$-planes in $\R^6$ as a special case:

\begin{theorem}\label{2-plane}
	Suppose $\SC_1$, $\SC_2$ and $\SC_3$ are three sets of $2$-planes in $\R^6$. Then for any $\epsilon > 0$, there exists a constant $C(\epsilon)$ such that the number of points at which at least three $2$-planes, one from each $\SC_i$, intersect and span $\R^6$ is bounded by $C(\epsilon)(|\SC_1||\SC_2||\SC_3|)^{\frac{1}{2}+\epsilon}$.
\end{theorem}

Comments: We can construct almost sharp examples for our main theorem by considering the axis-parallel example of the joints problem for lines in $R^{k}$ and turning each family of lines into a family of $\alpha_i$-planes.

Moreover, a simple modification of our proof of Theorem \ref{main} will almost answer Carbery's conjecture affirmatively, with a loss of $\epsilon$ in the power:

\begin{theorem}\label{multijoints with multiplicities}
	For any transversal collections $\mathcal{L}_1,\dots ,\mathcal{L}_n$ of $L_1,\dots ,L_n$ lines in $\R^n$ (transversal in the sense that, whenever $n$ lines, one from each collection, meet at a point, then they form a joint there) and any $\epsilon >0$, we have
\[
\sum\limits_{x}(\prod\limits_{1\leq i\leq n}N_i(x))^{\frac{1}{n-1}}=O((L_1\cdots L_n)^{\frac{1}{n-1}+\epsilon})
\] 
where $N_i(x)$ is the number of lines of $\mathcal{L}_i$ passing through $x$. The sum is taken over all joints.

\end{theorem} 

{\bf Acknowledgments.} I would like to thank my adviser Larry Guth, who has been very patient with me, for all the helpful conversations we had over this project. I would also like to thank David Jerison, for part of this research was done while I was supported on his NSF Grant DMS 1069225. 
\section{Main tools}
In Guth and Katz's celebrated work \cite{GK2} on the distinct distances conjecture, they developed the polynomial partitioning method which has been proven to be an extremely powerful tool in incidence geometry. Here we recall the statement of the partitioning theorem.

\begin{theorem}[Guth and Katz]\label{poly partition}
For each dimension $n$ and each degree $D$, the following holds. For any finite set $S\subset \R^n$, we can find a non-zero polynomial $P$ of degree at most $D$ so that $\R^n\backslash Z(P)$ is a union of disjoint open sets $O_i$, and for each of these sets,
\[
|S\cap O_i| \leq C_nD^{-n}|S|
\]
\end{theorem} 

Polynomial partitioning is a very useful divide and conquer argument. The point set $S$ is divided into a part in each cell plus a part in the surface $Z(P)$ and we estimate these contributions separately and then add up the results.

The polynomial partitioning technique can be applied to give new proofs of some classical results in incidence geometry. We refer the interested reader to \cite{KMS} which gives a very good exposition of the topic.

In \cite{G} Guth proved the following generalization of the partitioning theorem. Instead of a finite set of points, we consider a finite set of varieties. It could be a set of circles, a set of $d$-planes, etc. We would like to partition $\R^n$ with a degree $D$ polynomial $P$ so that each component of $\R^n\backslash Z(P)$ only intersects a small number of our varieties.    

\begin{theorem}[Guth]\label{poly partition for variety}
Suppose $\Gamma$ is a set of $k$-dimensional varieties in $\R^n$, each defined by at most $m$ polynomial equations of degree at most $d$. For any $D\geq 1$, there is a non-zero polynomial $P$ of degree at most $D$, so that each connected component of $\R^n\backslash Z(P)$ intersects at most $C(d,m,n)D^{k-n}|\Gamma|$ varieties in $\Gamma$.  
\end{theorem}

The fact that the power of $D$ in the above bound is $D^{k-n}$ is closely related to the following theorem proved by Barone and Basu bounding the number of components of $\R^n\backslash Z(P)$ a variety in $\Gamma$ can intersect. Here we state the version that appears in \cite{ST}.

\begin{theorem}[Barone and Basu]\label{BB}
Suppose $\gamma$ is a $k$-dimensional variety in $\R^n$ defined by at most $m$ polynomials equations each of degree at most $d$. If $P$ is a polynomial of degree at most $D$, then $\gamma$ intersects at most $C(d,m,n)D^k$ different components of $\R^n\backslash Z(P)$.
\end{theorem}

Suppose $P$ is a polynomial of degree $D$ and that $\R^n\backslash Z(P)$ consisted of $D^n$ connected components (cells) and that each cell intersects the same number of varieties $\gamma \in \Gamma$. Then Theorem \ref{BB} would imply that each connected component of $\R^n\backslash Z(P)$ intersected at most $C(d,m,n)D^k\cdot |\Gamma| \cdot D^{-n}=C(d,m,n)D^{k-n}|\Gamma|$ of our varieties. Theorem \ref{poly partition for variety} shows that there exists a polynomial $P$ of degree at most $D$ that achieves this bound.

Recently PVM Blagojevi\'{c}, ASD Blagojevi\'{c} and Ziegler \cite{BBZ} further extended the polynomial partitioning theorem to a setting with several families of varieties:

\begin{theorem}[PVM Blagojevi\'{c}, ASD Blagojevi\'{c} and Ziegler]\label{BBZ}
Let $j$ be an integer. For $1\leq i\leq j$, let $\Gamma _i$ be a finite set of $k_i$-dimensional varieties in $\R^n$, each of them defined by at most $m_i$ polynomial equations of degree at most $d_i$. Then for any $D\geq 1$ there exists a non-zero polynomial $P$ of degree at most $D$ such that each connected component of $\R^n\backslash Z(P)$ for every $1\leq j\leq i$ intersects at most $C(d_i,m_i,n)jD^{k_i-n}|\Gamma _i|$ varieties in $\Gamma _i$.
\end{theorem}

We will be using the above polynomial partitioning theorem for several sets of varieties to prove our main theorem. Notice that in our main theorem the number of sets of varieties is usually at most $n$ (because we are considering joints and we have the condition $\sum_{i=1}^{k} \alpha_i=n$) unless some of the varieties are just points, a trivial case we will deal with separately. Thus for our application the number of families of varieties $j$ is at most $n$ and can be covered by the constant $C(d_i,m_i,n)$.  

When applying the polynomial partitioning technique, one needs to choose the degree of the partitioning polynomial. In some cases one can choose $D$ to be a power of the number of points and this might give us good bound within each cell. But the problem is that this will make $Z(P)$ very complicated and it is hard to control the contributions from points lying on the surface. In \cite{ST}, Solymosi and Tao gave a modification of this argument using partitioning with degree $D$ equal a large constant, and using induction on the number of objects to control what happens in each cell. For incidences on the zero set of the partitioning polynomial, we are essentially looking at the incidence problem in an $(n-1)$-dimensional surface. Then we can project everything generically to $\R^{n-1}$. Thus we also need to use induction on the dimension of the problem. 

In summary, we will use a polynomial of degree equal to a large constant to partition the space so that each cell meets a bounded number of varieties from each set of our varieties (recall that we are given several sets of varieties and we are trying to bound the number of joints formed by them).We use induction on the number of varieties to control the contribution (in our case, the number of joints) from each cell and induction on dimensions to control the contribution from the zero set of the partitioning polynomial.

\section{Some algebraic geometry}
In this section we review some notation and facts that we need from algebraic geometry. We will be closely following Section $4$ of \cite{ST}, keeping the notation and theorems we need and omitting some of the proofs. 

\begin{definition} (Algebraic sets). Let $n\geq 1$ be a dimension. An \emph{algebraic set} in $\C^n$ is any set of the form 
\[
\{x\in \C^n:P_{1}(x)=\dots=P_{m}(x)=0 \}
\]
where $P_1,\dots ,P_m : \C^n \rightarrow \C$ are polynomials. An algebraic set is \emph{irreducible} if it cannot be expressed at the union of two strictly smaller algebraic sets. An irreducible algebraic sets will be referred to as an \emph{algebraic variety}, or \emph{variety} for short.

\end{definition}

The intersection of any subset of $\C^n$ with $\R^n$ will be referred to as the \emph{real points} of that subset. A \emph{real algebraic variety} is the real points $V_{\R}$ of a complex algebraic variety $V$.

We use the standard definition of the dimension dim($V$) and the degree deg($V$) for a variety $V$ in $\C^n$. The notions of dimension and degree of a real algebraic variety can be subtle if defined directly, but in our application each real algebraic variety will be associated with a complex one and we define the dimension and degree of the real variety to be that of the complex variety. 

The following two lemmas tell us that the degree and complexity (the number and degree of polynomials needed to define the variety) of a variety $V$ in $\C^d$ control each other.  
    
\begin{proposition}[Degree controls complexity]\label{D controls C}Let $V$ be an algebraic variety in $\C^n$ of degree at most $D$. Then we can write 
\[
V = \{ x\in \C^n : P_1(x)=\dots P_m(x)=0 \}
\] 
for some $m=O_{n,D}(1)$ and some polynomials $P_1,\dots ,P_m$ of degree at most $D$.
	
\end{proposition}

\begin{proposition}[Complexity controls degree]\label{C controls D}Let
\[
V = \{ x\in \C^n : P_1(x)=\dots P_m(x)=0 \}
\]
for some $m\geq 0$ and some polynomials $P_1,\dots ,P_m : \C^n \rightarrow \C$ of degree at most $D$. Then $V$ is the union of $O_{m,D,n}(1)$ varieties of degree $O_{m,D,n}(1)$.
\end{proposition}

A \emph{smooth point} of a $k$-dimensional variety $V$ is an element $p$ of $V$ such that $V$ can be locally described by a smooth $k$-dimensional complex manifold in a neighborhood of $p$. Points in $V$ that are not smooth will be called \emph{singular}. We let $V^{\text{smooth}}$ denote the smooth points of $V$, and $V^{\text{sing}}:=V\backslash V^{\text{smooth}}$ denote the singular points. 

The next proposition enables us to decompose a variety $V$ into smooth points of varieties:

\begin{proposition}\label{decomp into smooth}
Let $V$ be a $k$-dimensional algebraic variety in $\C^n$ of degree at most $D$. Then one can cover $V$ by $V^{\text{smooth}}$ and $O_{D,n}(1)$ sets of the form $W^{\text{smooth}}$, where $W$ is an algebraic variety in $V$ of dimension at most $k-1$ and degree $O_{D,n}(1)$. 
\end{proposition}

If $p\in V\subset \R^n$ is a smooth real point of a $k$-dimensional complex algebraic variety $V_{\C}\subset \C^n$, then it must have a $k$-dimensional complex tangent space. However, its real tangent space may have dimension smaller than $k$. This may cause some trouble as our definition of joint requires that the tangent space at a joint of a variety should have the same dimension as the variety itself. The following proposition will help us resolve this issue.

\begin{proposition}\label{full tangent space}
Let $V$ be a $k$-dimensional algebraic variety in $\C^n$ of degree at most $D$. Then at least one of the following statement is true:
\begin{enumerate}
\item The real points $V_{\R}$ of $V$ are covered by the smooth points $W^{\textnormal{smooth}}$ of $O_{D,n}(1)$ algebraic varieties $W$ of dimension at most $k-1$ and degree $O_{D,n}(1)$ which are contained in $V$.
\item For every smooth real point $p\in V_{\R}^{\textnormal{smooth}}$ of $V$, the real tangent space is $k$-dimensional. In particular, $V_{\R}^{\textnormal{smooth}}$ is $k$-dimensional. 
\end{enumerate}
\end{proposition}  

As pointed out at the end of Section $4$ of \cite{ST}, we may assume that each of the varieties occurring in case $1$ obey the properties stated in case $2$. 

To prove the main theorem, we need to project varieties to subspaces of $\C^n$ or consider intersections of two varieties. Thus we need the following propositions to describe what happens in these two situations.

\begin{proposition}[Projection]\label{projection}
	Let $V$ be a $k$-dimensional algebraic variety in $\C^n$ of degree at most $D$. $\pi$ is a generic projection to a hyperplane which we will identify with $\C^{n-1}$. Then $\pi(V)$ is covered by $O_{D,n}(1)$ algebraic varieties of dimension at most $k$ and degree at most $O_{D,n}(1)$ in $\C^{n-1}$.   
\end{proposition}

\begin{proposition}[Intersection]\label{intersection}
	Suppose $V_1$ and $V_2$ are $k_1$-dimensional and $k_2$-dimensional algebraic varieties of degree $D_1$ and $D_2$ respectively in $\C^n$. If $V_1$ is not contained in $V_2$, then their intersection $V_1\cap V_2$ will be an algebraic set containing $O_{D_1,D_2,n}(1)$ varieties whose dimension is strictly smaller than $k_1$ and degree at most $O_{D_1,D_2,n}(1)$. 
\end{proposition}

\section{Outline of the proof and a special case}
In our main theorem, if we take $\alpha _1=\alpha_2=\alpha_3=1$ and $k=n=3$ with all the varieties being lines, we recover the following version of the (multi)joints problem for lines:
\begin{theorem}\label{joints with epsilon}
	Suppose we have three non-empty sets of lines $\LC _1$, $\LC_2$ and $\LC_3$. A point $x$ is a joint if there are $l_1\in \LC_1$, $l_2\in \LC_2$ and $l_3\in \LC_3$ such that they all pass $x$ and they are not coplanar. Then for any $\epsilon >0$, there exists a constant $C$ such that the number of joints is bounded by $C|\LC_1|^{1/2+\epsilon}|\LC_2|^{1/2+\epsilon}|\LC_3|^{1/2+\epsilon}$. In particular, if $|\LC_1|=|\LC_2|=|\LC_3|=N$, then the bound is $CN^{3/2+3\epsilon}$.
\end{theorem}

We would like to give an outline of the proof of Theorem \ref{main} first and then we will prove Theorem \ref{joints with epsilon} to illustrate the method.

Here is the outline: first we do polynomial partitioning for all of our varieties using a polynomial of degree $D$ where $D$ that will be chosen later. There will be roughly $D^n$ cells and within each cell, the number of $\alpha _i$-dimensional varieties from $\SC_i$ is bounded by $\sim |\SC_i|D^{\alpha _i -n}$. Now if we use induction on the number of varieties, we can show that the number of joints within each cell is bounded by 
\begin{align*}
& C\prod\limits_{1\leq i\leq k}(|\SC_i|D^{\alpha _i -n})^{1/(k-1)+\epsilon} \\
=& C(\prod\limits_{1\leq i\leq k} |\SC_i|)^{1/(k-1)+\epsilon}D^{(\sum_{i=1}^{k} \alpha _i -nk)(1/(k-1)+\epsilon)} \\ 
= & C(\prod\limits_{1\leq \i\leq k}|\SC_i|)^{1/(k-1)+\epsilon}D^{(n-nk)(1/(k-1)+\epsilon)}\\
= &C(\prod\limits_{1\leq i\leq k} |\SC_i|)^{1/(k-1)+\epsilon}D^{-n}D^{(n-nk)\epsilon}
\end{align*}
 Since in total we have roughly $D^n$ cells, the number of joints from all cells is bounded by $D^nC(\prod |\SC_i|)^{1/(k-1)+\epsilon}D^{-n}D^{(n-nk)\epsilon}=C(\prod |\SC_i|)^{1/(k-1)+\epsilon}D^{(n-nk)\epsilon}$. $k$ is at least $2$, so $(n-nk)\epsilon$ is smaller than $0$ and the exponent of $D$ is negative. We choose $D$ big enough so that the number of joints in cells is well controlled. 

Next let's look at joints on the zero set of the partitioning polynomial $Z(P)$. There are two cases: a joint $x$ could be a smooth point of $Z(P)$ or a singular point of $Z(P)$. If it is the former case, then we claim that the varieties $S_i\in\SC_i$ that form the joint $x$ cannot be all contained in $Z(P)$. If they are all contained in $Z(P)$, then their tangent spaces at $x$ will all be contained in the tangent space of $Z(P)$ at $x$, which is an $(n-1)$-dimensional space because we assumed that $x$ is a smooth point of $Z(P)$. Thus the tangent spaces do not span $\R^n$ which is a contradiction to the definition of joints. Hence at least one of $S_i$'s should not be contained in $Z(P)$. Let's now consider the joints for which the variety from $\SC_1$ is not contained in $Z(P)$. We intersect all varieties in $\SC_1$ with $Z(P)$ and get a new set of varieties $\SC_1'$. For simplicity let's assume that all varieties in $\SC_1'$ have dimension $\alpha_1'=\alpha_1 -1$. We then perform a generic projection $\pi$ from $\R^n$ to $\R^{n-1}$. Again for simplicity we assume that each variety becomes a variety of same dimension and degree in $\R^{n-1}$ (notice that actually the projection of a variety might not be a variety anymore and the degree and dimension might change; we will deal with this minor issue in the proof of Theorem \ref{main}. Similar issues exist when we consider the intersection of two varieties.). Suppose $x$ is a joint formed by $S_i\in \SC_i$ and $S_1$ is not contained in $Z(P)$. Then we would like to claim that after the projection $\pi(x)$ is a joint formed by $\pi(S_1')$, $\pi(S_2)$...$\pi(S_k)$. This is morally true, and the only problem is that $x$ might be a singular point of $\pi(S_1')$. Let's ignore this in this outline. Now we would like to use induction on $n$ to bound the number of joints formed by varieties in $\SC_1'$, $\SC_2$,...,$\SC_k$ (notice that $\alpha_1'+\alpha_2+...+\alpha_k=n-1$) by $C(|\SC_1'|\prod\limits_{2\leq i\leq k} |\SC_i|)^{1/(k-1)+\epsilon}=C(\prod\limits_{1\leq i\leq k} |\SC_i|)^{1/(k-1)+\epsilon}$. This is how we control the number of the joints that are smooth points of $Z(P)$.

Finally let's bound the number of joints that are singular points of $Z(P)$. We can just replace $Z(P)$ with the set of singular points of $Z(P)$ (denote it by $Z_1$) and run the above argument again: we look at two cases where the joint is smooth or singular points of $Z_1$. If it is smooth, then we know that the varieties cannot be all contained in $Z_1$, so we look at intersections of $Z_1$ with our varieties. If it is singular, then we have to look at $Z_2$ which is the set of singular points of $Z_1$. Eventually there would be a $Z_j$ which a just a set of isolated points. The number of points in $Z_j$ is bounded by a constant only depending on $n$ and $d$, so it is well controlled. 

Another way to look at the above argument is the following. After we do polynomial partitioning, we bound the number of joints in cells as we did in the above argument. For joints on the partitioning surface $Z(P)$, we can use Proposition \ref{full tangent space} to decompose $Z(P)$ into smooth points of varieties. The number and degree of those varieties are all constants depending on $D$ and $n$. Now we only need to consider the number of joints that are smooth points of a variety and we can use the above argument for the smooth-point case.     

This is how we will prove Theorem \ref{main}.  In summary, first we use polynomial partitioning. Then for joints in cells we use induction on the number of varieties. As long as the degree of our partitioning polynomial is big enough, the number of such joints is well controlled. For joints on the zero set of the partitioning polynomial, we use induction on dimensions of our sets of varieties. We will consider the intersection of our varieties with another variety and the singular points of our varieties, for both of which the dimensions of our varieties will go down.    

Now we use the above argument to prove Theorem \ref{joints with epsilon}. In this special case, the main part of the proof is the induction on the number of lines. For the induction-on-dimension part, the dimension of lines is just one and if the dimension of lines goes down, they will be sets of points in which case the bound is very easy to prove.  

\begin{proof}[Proof of Theorem \ref{joints with epsilon}]
Fix $\epsilon>0$. We would like to find a $C_0$ depending on $\epsilon$ so that the bound holds. Through this proof any other constant that only depends on $\epsilon$ will be denoted by $C$. 

We would like to use induction on the number of lines in each of $\LC_i$ and show that the induction can be closed if $C_0$ is big enough. Let's look at the base case first when one of the families $\LC_i$ only has a small number of lines. WLOG we assume that $|\LC_1|\leq 100$. For a point $x$ to be a joint, it must lie on one of the lines in $\LC_1$. Thus we can fix a line $l_1\in \LC_1$ and only consider joints on $l_1$, then we multiply the bound we get by $100$ to bound the total number of joints.

$l_1$ can have at most $|\LC_2|$ intersections with lines in $\LC_2$ and at most $|\LC_3|$ intersections with lines in $\LC_3$, so the number of joints on $l_1$ is bounded by $\min \{|\LC_2|,|\LC_3| \}\leq (|\LC_2||\LC_3|)^{\frac{1}{2}}$. Thus the bound holds for the base case (when one of the sets has less than $100$ lines) if we choose our $C_0$ to be at least $100$. 

Now let's assume the bound is true if the number of lines in each $\LC_i$ is smaller. We use a polynomial of degree $D$ (which will be chosen later) to partition the space so that in each cell the number of lines from $\LC_i$ is bounded by $C|\LC_i|D^{-2}$ (here $C$ is the constant in the partitioning theorem for lines). Now from induction hypothesis we know that in each cell the number of joints is bounded by $C_0C(|\LC_1|D^{-2})^{1/2+\epsilon}(|\LC_2|D^{-2})^{1/2+\epsilon}(|\LC_3|D^{-2})^{1/2+\epsilon}=C_0C(|\LC_1||\LC_2||\LC_3|)^{1/2+\epsilon}D^{-3-6\epsilon}$. There are at most $CD^3$ cells (this $C$ is another constant, different from the one in the partitioning theorem), so in total the number of joints in cells is bounded by $C_0C(|\LC_1||\LC_2||\LC_3|)^{1/2+\epsilon}D^{-6\epsilon}$. Now we can choose our $D$ to be big enough so that $CD^{-6\epsilon}$ is much smaller than $1$ (say smaller than $1/10$). Notice that $D$ only depends on $\epsilon$. 

Next we look at joints on the zero set $Z(P)$ of the partitioning polynomial $P$. We first consider the joints that are smooth points of $Z(P)$. Let $x$ be such a joint formed by $l_1\in\LC_1$, $l_2\in\LC_2$ and $l_3\in\LC_3$. Then we know that $l_1$, $l_2$ and $l_3$ cannot all be contained in $Z(P)$. From the vanishing lemma, we know that if a line is not contained in $Z(P)$, then it can only have $D=$ deg$(P)$ intersections with $Z(P)$. This implies that the number of joints for which the line from $\LC_1$ is not contained in $Z(P)$ is bounded by $D|\LC_1|$. From the definition of joint we also see that the number of joints is bounded by $|\LC_2||\LC_3|$. With these two bounds, the number of joints is also bounded by their geometric mean which is $D^{1/2}(|\LC_1||\LC_2||\LC_3|)^{1/2}$. $D$ is the degree of our partitioning polynomial and it only depends on $\epsilon$. The number of joints for which the line from $\LC_2$ or $\LC_3$ is not contained in $Z(P)$ can be dealt with similarly. Thus as long as $C_0$ is greater than $100D^{1/2}$, the number of such joints is bounded by $\frac{1}{10}C_0(|\LC_1||\LC_2||\LC_3|)^{1/2+\epsilon}$ 

Finally we look at joints that are singular points of $Z(P)$. We denote the set of singular points of $Z(P)$ by $Z_1$. $Z_1$ is an algebraic curve whose degree $D_1$ only depends on $D$ (thus only depends on $\epsilon$). Suppose $x$ is such a joint formed by $l_1$, $l_2$ and $l_3$. If $x$ is a smooth point of $Z_1$, then one of (actually two of) the lines $l_i$ should not be contained in $Z_1$. Again from the vanishing lemma We know that if a line is not contained in $Z_1$ then it can only intersect $Z_1$ in $D_1$ points and using an argument similar to the one in the above paragraph we see that the number of such joints is well controlled (as long as $C_0$ is greater than $100D_1^{1/2}$). If $x$ is a singular point of $Z_1$, then we can bound the number of such joints by the total number of singular points of $Z_1$, which is a number that only depends on $D_1$ (thus only depends on $\epsilon$). Summing up the number of joints in all above cases closes the induction and this  finishes our proof. 
\end{proof}

\section{Proof of main theorem}

As stated in the outline of the proof, we have to consider joints that are on the partitioning surface. Then we have to look at joints that are singular points of that surface and the set of singulars points is algebraic set of lower dimensions. Thus we see that we need to consider the joints problem with the additional condition that all the joints are on a variety of a certain dimension. This motivates us to prove the following (a little bit weaker) version of Theorem \ref{main}:

\begin{theorem}\label{induction}
	
Suppose $\SC _i$ is a non-empty set of $\alpha _i$-dimensional algebraic varieties in $\R ^n$ for $1\leq i\leq k$ ($k\geq 2$) and $\sum\limits_{1\leq i\leq k} \alpha _i=n\geq 2$, $\alpha _i<n$. We also assume that each variety is defined by $m$ polynomial equations each of degree at most $d$. $V$ is an algebraic variety that has dimension $\tilde{\alpha}<n$ and is defined by $m$ polynomial equations each of degree at most $\tilde d$. Then for any $\epsilon >0$, there exists a constant $C(n,m,d,\tilde{d},\epsilon)$ such that the number of joints formed by $\SC_i$ that are on $V$ is bounded by $C(n,m,d,\tilde{d},\epsilon)(\prod\limits_{1\leq i\leq k} |\SC_i|)^{1/(k-1)+\epsilon}$

\end{theorem}

Using the above theorem together with polynomial partitioning, we can prove our main theorem :

\begin{proof}[Proof of Theorem \ref{main}]
We denote the constant from Theorem \ref{induction} by $C_0$, the constant we are trying to find for Theorem \ref{main} by $C_1$, and all other constants that only depend on $n$, $m$, $d$ or $\epsilon$ by $C$.

We use induction on the total number of varieties from $\SC_i$. We will choose $C_1$ big enough (but only depending on $n$, $m$, $d$ and $\epsilon$) so that our induction will close. We deal with the base case first. Assume that the size of at least one of the families of varieties is smaller than $100$. WLOG we assume $|\SC_1|\leq 100$. Then we fix a variety $s_1$ in $\SC_1$ and look at how many joints are on it. Later on we only need to multiply this number by $100$ to get the bound for the total number of joints. From the definition of joints and B\'{e}zout theorem, we see that if we fix another variety $s_2$ in $\SC_2$, then the number of joints that are on both $s_1$ and $s_2$ is bounded by $C$. Thus the number of joints that are on $s_1$ is bounded by $C|\SC_2|$. Similarly it is bounded by $C|\SC_i|$ for any $2\leq i\leq k$. Thus it is bounded by $C((\prod_{i=2}^{n}|\SC_i|))^{\frac{1}{k-1}}$. Multiplying this number by $100$, we get that the total number of joints is at most $100C(\prod_{i=2}^{n}|\SC_i|)^{\frac{1}{k-1}}$. As long as we choose $C_1$ to be greater than $100C$, the bound will hold for our base case (where the size of some family of varieties is at most $100$). 

We assume that the bound holds when the size of each family of varieties is smaller than it is now. We would like to use the zero set of a polynomial of degree $D$ (which will be chosen later) to partition $\R^n$ so that in each cell the number of varieties from $\SC_i$ is bounded by $C|\SC_i|D^{\alpha_i-n}$, but here is a small issue. The number of varieties in each cell might also depend on the number of families of varieties (this makes sense because you can't expect a polynomial with degree $D$ to do anything when there is a huge number of families of varieties to partition for). This is just a technical issue. If our varieties all have positive dimension, then the number of families is bounded by $n$ (see our definition of joints) and there won't be any problem. Otherwise, suppose $\SC_1$ is actually a family of sets of points. Then any joint $x$ formed by $\SC_1, \SC_2,\dots ,\SC_k$ is also a joint formed by $\SC_2,\dots ,\SC_k$. Suppose we know that $\SC_2,\dots , \SC_k$ can form at most $C_1(\prod_{i=2}^{k}|\SC_i|)^{\frac{1}{k-2}+\epsilon}$. Since $\SC_1$ consists of sets of points and each set has at most $C$ points, we see that the number of joints is also bounded by $C|\SC_1|$. Now we have two bounds for the number of joints formed by $\SC_1, \SC_2, \dots, \SC_k$: $A=C|\SC_1|$ and $B=C_1(\prod_{i=2}^{k}|\SC_i|)^{\frac{1}{k-2}+\epsilon}$. Thus the number of joints is also bounded by the "mean" of them, which is  $A^{\frac{1}{k-1}}B^{\frac{k-2}{k-1}}=(C|\SC_1|)^{\frac{1}{k-1}}(C_1(\prod_{i=2}^{k}|\SC_i|)^{\frac{1}{k-2}+\epsilon})^{\frac{k-2}{k-1}}\leq C_1(|\SC_1||\SC_2|\dots |\SC_k|)^{\frac{1}{k-1}+\epsilon}$. The last inequality holds as long as $C_1\geq C$. This argument shows that we only need to consider the case where all of our varieties have positive dimension.

Now we can use the zero set of a polynomial of degree $D$ (which will be chosen later) to partition $\R^n$ so that in each cell the number of varieties from $\SC_i$ is bounded by $C|\SC_i|D^{\alpha_i-n}$. We assume that all of our varieties have positive dimension, so the number of families of varieties is bounded by $n$ and the constant $C$ only depends on $n$, $m$, and $d$. Later on we will choose a big $D$ so that $CD^{-1}$ is smaller than $1$. Thus we can use induction and conclude that the number of joints in each cell is bounded by 
\[
C_1C(\prod_{i=1}^{k} |\SC_i|D^{\alpha_i-n})^{1/(k-1)+\epsilon}=C_1C(\prod_{i=1}^{k} |\SC_i|)^{1/(k-1)+\epsilon}D^{-n}D^{n(1-k)\epsilon}.
\]
There are $CD^n$ cells in total, so the number of joints in cells in total is $C_1CD^{n(1-k)\epsilon}\\ (\prod_{i=1}^{k} |\SC_i|)^{1/(k-1)+\epsilon}$. We can choose $D$ such that $CD^{n(1-k)\epsilon}\leq 1/10$. This is how we bound the number of joints in cells.

Now we can use Theorem \ref{induction} to bound the number of joints on the zero set of the partitioning surface $Z(P)$. $Z(P)$ can be covered by $C$ varieties, each of which is defined by at most $C$ polynomials of degree at most $C$. Theorem \ref{induction} tells us that the number of joints on each of those varieties is bounded by $C(n,C,d,C,\epsilon)(\prod_{i=1}^{k} |\SC_i|)^{1/(k-1)+\epsilon}$. Thus we can choose $C_1$ to be bigger than $10C(n,C,d,C,\epsilon)$. For this $C_1$, the sum of number of joints in cells and on the surface would be smaller than $(\frac{1}{10}C_1+\frac{1}{10}C_1)(\prod_{i=1}^{k} |\SC_i|)^{1/(k-1)+\epsilon}<C_1(\prod_{i=1}^{k} |\SC_i|)^{1/(k-1)+\epsilon}$

\end{proof}

Let's prove Theorem \ref{induction}.

\begin{proof}
We would like to use induction on $n$ and $\tilde{\alpha}$. As in the proof of Theorem \ref{main}, we assume that all $\alpha_i>0$. Let's first look at some base cases.

When $n=2$, $k$ must be $2$. Thus from the definition of joints and B\'{e}zout Theorem we see that the number of joints is bounded by $C|\SC_1||\SC_2|=C(|\SC_1||\SC_2|)^{\frac{1}{k-1}}$.
	
When $\tilde{\alpha}=0$, the number of points on the $0$-dimensional variety $V$ is bounded by $C$. So the statement is automatically true. 

Suppose the statement is true when the values of $n$ and $\tilde{\alpha}$ are smaller (when both of them are at most what they are now and at least one of them is strictly smaller). We want to prove that there is $C_0$ that only depends on $n$, $m$, $d$, $\tilde{d}$ and $\epsilon$ such that the number of joints is bounded by $C_0(\prod_{i=1}^{k} |\SC_i|)^{1/(k-1)+\epsilon}$. We use $C$ to denote any other number that only depends on $n$, $m$, $d$, $\tilde{d}$ and $\epsilon$.
 
First, from induction hypothesis and Proposition \ref{full tangent space} we only have to consider joints that are smooth points of $V$ and the real tangent spaces of $V$ at those points are $\tilde{\alpha}$-dimensional.   

If $x$ is a joint formed by $S_i\in \SC_i$ and it is a smooth point of $V$, then from the definition of joints we see that at least one of $S_i$ is not contained in $V$. Let's focus on the case where $S_1$ is not contained in $V$. There are $k-1$ other cases, but $k$ is at most $n$ and we can bound those cases similarly and add them up.

Now we consider intersections of varieties in $\SC_1$ (that are not contained in $V$) with $V$. Their intersections are some algebraic sets of dimension smaller than $\alpha_1$. We use Proposition \ref{D controls C}, \ref{C controls D} and \ref{full tangent space} to cover those algebraic sets by the smooth points of $C|\SC_1|$ varieties defined by at most $C$ polynomials each of degree at most $C$. Let's denote those varieties of dimension $q$ by $\SC_1^q$, $0\leq q\leq \alpha_1-1$. Suppose $x$ is a joint formed by $S_i$ and $S_1$ is not contained in $V$. Then $x\in S_1\cap V$ and $x$ must be a smooth point of some variety $S_1^q$ in $\SC_1^q$ for some $q$. Now we project everything down to $\R^{n-\alpha_1+q}$ generically (we denote the projection by $\pi$). We would like to claim that $\pi(x)$ is a joint formed by $\pi(S_1^q), \pi(S_2),\dots \pi(S_k)$, but there is again a technical issue. The projection of a variety might not be a variety anymore. Thus we need to use Proposition \ref{projection} to cover the projection of a variety with $C$ varieties of equal or smaller dimension and degree $C$. In this way each family of varieties, $\SC_1^q, \SC_2,\dots, \SC_k$, becomes a new family of varieties in $\R^{n-\alpha_1+q}$. We denote them by (with an abuse of notation) $\pi(\SC_1^q),\pi(\SC_2),...,\pi(\SC_k)$. We see that $\pi(x)$, which is the projection of $x$, must be a joint formed by some variety in $\pi(\SC_1^q)$ that has dimension $q$ and some varieties in $\pi(\SC_i)$ that have dimension $\alpha_i$, $i\geq 2$. Thus to bound the number of joints like $x$, we only need to bound the number of joints formed by varieties in $\pi(\SC_1^q)$ that have dimension $q$ and varieties in $\pi(\SC_i)$ that have dimension $\alpha_i$. Now we can use induction hypothesis since we are bounding the number of joints formed by varieties in $\R^{n-\alpha_1+q}$. The number of varieties in $\pi(\SC_1^q)$ that have dimension $q$ is bounded by $C|\SC_1|$ and the number of varieties in $\SC_i$ that have dimension $\alpha_i$ is bounded by $C|\SC_i|$. We have proven that Theorem \ref{induction} would imply Theorem \ref{main}; in particular, if Theorem \ref{induction} is true when the dimension of our ambient space is no greater than $n-\alpha_1+q$, then Theorem \ref{main} is true in the same ambient space. Thus our induction hypothesis will imply that the number of joints formed by varieties in $\pi(\SC_1^q)$ that have dimension $q$ and varieties in $\pi(\SC_i)$ that have dimension $\alpha_i$ can be bounded by $C(\prod_{i=1}^{k} |\SC_i|)^{\frac{1}{k-1}+\epsilon}$. Here the constant $C$ might depend on the following parameters: the dimension $n$, the number and degree of polynomials needed to define our varieties in $\pi(\SC_1^q)$ and $\pi(\SC_i)$, and $\epsilon$. Notice that the number and degree of polynomials needed to define our varieties might be bigger than $m$ and $d$ (because we considered intersections and projections), but they are bounded by $C$ (a number depending on $n$, $m$, $d$ and $\tilde{d}$). Now we can bound the number of joints in the other $k-1$ cases (where varieties from $\SC_i$ are not contained in $V$ for $i\geq 2$) similarly and the total number of joints is still bounded by $C(\prod_{i=1}^{k} |\SC_i|)^{\frac{1}{k-1}+\epsilon}$. We can choose $C_1$ to be much greater than this particular $C$ and we are done.
\end{proof}
 
Now let's prove Theorem \ref{multijoints with multiplicities}. The argument is very similar to the one used in the proof of Theorem \ref{main}.

\begin{proof}[Proof of Theorem \ref{multijoints with multiplicities}]
We use induction on the dimension $n$ and the number of lines in each family. Let's look at the base cases first. 

When $n=2$, the left-hand side of the inequality is the number of intersections between lines in $\LC_1$ and lines in $\LC_2$ which is at most $L_1L_2$. So the statement is true. When one of $L_i$ is smaller than $100$, WLOG we assume that $L_1\leq 100$. If $x$ is a joint, then $x$ must lie on one of the lines in $\LC_1$. The number of intersections between lines in $\LC_1$ and lines in $\LC_i$ is at most $100L_i$ for $2\leq i\leq n$. Thus $\sum\limits_{x} N_i(x)\leq 100L_i$ for $2\leq i\leq n$. Now we can use H\"{o}lder's inequality to bound the left-hand side:
\begin{align*}
& \sum\limits_{x}(\prod\limits_{1\leq i\leq n}N_i(x))^{\frac{1}{n-1}} \\
\leq & L_1^{\frac{1}{n-1}}
\sum\limits_{x}(\prod\limits_{2\leq i\leq n}N_i(x))^{\frac{1}{n-1}} \\ 
\leq & L_1^{\frac{1}{n-1}}\prod\limits_{2\leq i\leq n}(\sum\limits_{x}N_i(x))^{\frac{1}{n-1}}\\
\leq & L_1^{\frac{1}{n-1}}\prod\limits_{2\leq i\leq n}(100L_i)^{\frac{1}{n-1}}\\
\leq & 100\prod\limits_{1\leq i\leq n}(L_i)^{\frac{1}{n-1}}
\end{align*}
Thus the statement is true in this case.

Suppose the statement is true when the dimension $n$ is smaller, or is the same and the number of lines in each family is smaller. Theorem \ref{BBZ} enables us to find a partitioning polynomial $P$ of degree $D$ s.t. each connected component $O_i$ of $\R^n\backslash Z(P)$ intersects at most $CL_iD^{1-n}$ lines from $\LC_i$. Now we can use induction hypothesis to control the sum over joints that are in $O_i$:
\[ \sum\limits_{x\in O_i}(\prod\limits_{1\leq i\leq n}N_i(x))^{\frac{1}{n-1}}
\leq C_1(\prod\limits_{1\leq i\leq n}(CL_iD^{1-n}))^{\frac{1}{n-1}+\epsilon}
\leq C_1C(\prod\limits_{1\leq i\leq n}L_i)^{\frac{1}{n-1}+\epsilon}D^{-n-n(n-1)\epsilon}  
\]
There are less than $CD^n$ cells in total, so the contribution of joints in cells to the left hand side is at most $C_1CD^{-n(n-1)\epsilon}(\prod\limits_{1\leq i\leq n}L_i)^{\frac{1}{n-1}+\epsilon}$. As long as we choose $D$ big enough (depending on $C$ and $\epsilon$), this is smaller than $\frac{1}{2}C_1(\prod\limits_{1\leq i\leq n}L_i)^{\frac{1}{n-1}+\epsilon}$.

Now let's look at $\sum\limits_{x\in Z(P)}(\prod\limits_{1\leq i\leq n}N_i(x))^{\frac{1}{n-1}}$. We use Proposition \ref{full tangent space} to decompose $Z(P)$ into smooth points of $C_2$ varieties $V_i$ of dimension at most $n-1$ and degree at most $C$. We denote the set of lines in $\LC_i$ that pass $x$ by $\LC_i(x)$. If $x$ is a smooth point of $V_1$, we claim that there exists an $i$ s.t. no line in $\LC_i(x)$ is contained in $V_1$. This is because otherwise we could pick one line from each $\LC_i(x)$ and they are all contained in the tangent space of $V_1$ at $x$, violating our assumption that $\LC_1,\dots, \LC_n$ are transversal collections of lines. We first look at the set of joints where lines in $\LC_1(x)$ are not contained in $V_1$ and denote it by $J_1$. Then $\sum\limits_{x\in J_1}N_1(x)\leq CL_1$ because a line that is not contained in $V_1$ can have at most $C$ intersections with it. Now we use H\"{o}lder's inequality and have 
\[
 \sum\limits_{x\in J_1}(\prod\limits_{1\leq i\leq n}N_i(x))^{\frac{1}{n-1}} \leq
 ((\sum\limits_{x\in J_1}N_1(x))(\sum\limits_{x\in J_1}(\prod\limits_{2\leq i\leq n}N_i(x))^{\frac{1}{n-2}})^{n-2})^{\frac{1}{n-1}}
\]

We can bound the second factor $\sum\limits_{x\in J_1}(\prod\limits_{2\leq i\leq n}N_i(x))^{\frac{1}{n-2}}$ by projecting lines in $\LC_i$ ($2\leq i\leq n$) generically to $\R^{n-1}$ and using our induction hypothesis when the dimension is $n-1$. 
\[
\sum\limits_{x\in J_1}(\prod\limits_{2\leq i\leq n}N_i(x))^{\frac{1}{n-2}}\leq C(\prod\limits_{2\leq i\leq n}L_i)^{\frac{1}{n-2}+\epsilon}
\]
Thus we have 
\begin{align*}
& \sum\limits_{x\in J_1}(\prod\limits_{1\leq i\leq n}N_i(x))^{\frac{1}{n-1}} \\ \leq
& ((\sum\limits_{x\in J_1}N_1(x))(\sum\limits_{x\in J_1}(\prod\limits_{2\leq i\leq n}N_i(x))^{\frac{1}{n-2}})^{n-2})^{\frac{1}{n-1}} \\ \leq
& (CL_1(C(\prod\limits_{2\leq i\leq n}L_i)^{\frac{1}{n-2}+\epsilon})^{n-2})^{\frac{1}{n-1}} \\ \leq
& C(\prod\limits_{1\leq i\leq n}L_i)^{\frac{1}{n-1}+\frac{n-2}{n-1}\epsilon} \\ \leq
& \frac{1}{2nC_2}C_1(\prod\limits_{1\leq i\leq n}L_i)^{\frac{1}{n-1}+\epsilon}
\end{align*}
The last inequality holds when we choose $C_1$ to be greater than $2nCC_2$.

There are also contributions from $J_i$ ($2\leq i\leq n$), so the sum taken over all joints that are smooth points of $V_1$ is bounded by $n$ times the above bound, which is $\frac{1}{2C_2}C_1(\prod\limits_{1\leq i\leq n}L_i)^{\frac{1}{n-1}+\epsilon}$. There are $C_2$ varieties in total, so the contributions from joints on $Z(P)$ is bounded by $\frac{1}{2}C_1(\prod\limits_{1\leq i\leq n}L_i)^{\frac{1}{n-1}+\epsilon}$. Since we have already bounded the contributions from joints in cells by $\frac{1}{2}C_1(\prod\limits_{1\leq i\leq n}L_i)^{\frac{1}{n-1}+\epsilon}$, the induction is closed.

\end{proof}

\end{document}